 \documentclass[final,3p,10pt,times]{elsarticle}

\usepackage{amssymb}
\usepackage{amsthm}
\usepackage[english]{babel} \usepackage{latexsym,amsfonts,amsmath,mathrsfs,color}
\usepackage{mathpazo} \usepackage[pdftex]{hyperref}
\numberwithin{equation}{section}  \makeatletter 

\def\ps@pprintTitle{%
 \let\@oddhead\@empty
 \let\@evenhead\@empty
 \def\@oddfoot{\centerline{\thepage}}%
 \let\@evenfoot\@oddfoot}

\@addtoreset{equation}{section}
\newcommand{\norm}[1]{\left\Vert#1\right\Vert}
\newcommand{\scal}[1]{\left<#1\right>}
\newtheorem{theorem}{Theorem}[section]
\newtheorem{proposition}[theorem]{Proposition}
\newtheorem{defn}[theorem]{Definition}
\newtheorem{lemma}[theorem]{Lemma}

\newtheorem{remark}[theorem]{Remark}
\newcommand{\Hq}{\mathbb H}
\newcommand{\Sq}{\mathbb S}

\newcommand{\R}{\mathbb{R}}      

\newcommand{\C}{\mathbb{C}}

\begin{document}
\begin{frontmatter}
\title{A quaternionic analogue of the Segal-Bargmann transform }
\author{K. Diki}   
\author{A. Ghanmi}     
 \address{P.D.E. and Spectral Geometry,
          Laboratory of Analysis and Applications - URAC/03,
          Department of Mathematics, P.O. Box 1014,  Faculty of Sciences,
          Mohammed V University in Rabat, Morocco}
\begin{abstract}
The Bargmann-Fock space of slice hyperholomorphic functions is recently introduced by Alpay, Colombo, Sabadini and Salomon. In this paper, we reconsider this space and present a direct proof of its independence of the slice. We also introduce a quaternionic analogue of the classical Segal-Bargmann transform and discuss some of its basic properties. The explicit expression of its inverse is obtained and the connection to the left one-dimensional quaternionic Fourier transform is given.
\end{abstract}




\end{frontmatter}

\section{Introduction} 
In 2014,  Alpay, Colombo, Sabadini and Salomon have introduced a quaternionic analogue of the classical Bargmann-Fock space in the setting of the new theory of slice regular functions (see \cite{AlpayColomboSabadini2014}). More precisely, they considered
  $$ \mathcal{F}^{2,\nu}_{slice}(\Hq) = \mathcal{SR}(\Hq) \cap L^2(\C_I;e^{-\nu|q|^2}d\lambda_I(q)),$$
  where $ \mathcal{SR}(\Hq)$ is the space of (slice) regular $\Hq$-valued functions on $\Hq$ and $\C_I$ is a slice in $\Hq$ associated to given $I\in \mathbb{S}=\lbrace{q\in{\Hq};q^2=-1}\rbrace$. They proved that $\mathcal{F}^{2,\nu}_{slice}(\Hq)$ is independent of the choice of the imaginary unit $I\in \mathbb{S}$ and is a reproducing kernel Hilbert space.

 In the present paper, we will reconsider $\mathcal{F}^{2,\nu}_{slice}(\Hq)$ and review some of its properties, although some new results are also obtained.
 Moreover, we introduce and study in some details a quaternionic analogue of the classical Segal-Bargmann transform that maps isometrically the classical Hilbert space on the real line to the slice hyperholomorphic Bargmann-Fock space.  We also yield an integral representation of its inverse and discuss some interesting results connecting this transform to the left one-dimensional quaternionic Fourier transform \cite{EllLeBihanSangwine2014}
    $$\mathcal{F}_I(\psi)(x) := \int_{\R} e^{I x y} \psi(y) dy  ; \quad I^2=-1 .$$

  To present these ideas and others, we are going to adopt the following structure: in Section 2 we collect some useful backgrounds from the theory of slice regular functions. Section 3 consists of a direct proof of the space $\mathcal{F}^{2,\nu}_{slice}(\Hq)$ being independent of the slice. Section 4
    is devoted to the introduction and the investigation of the basic properties of the quaternionic Segal-Bargmann transform. Its inverse transform
    and connection to a like one-dimensional quaternionic Fourier transform are also discussed in this section. In the last section we collect some concluding remarks on the full hyperholomorphic Bargmann-Fock space.


\section{Preliminaries on slice regular functions}
The theory of quaternionic slice regular functions has been introduced quite recently by Gentilli and Struppa in \cite{GentiliStruppa06} (see also \cite{GentiliStruppa07}). Since then, it was extensively studied and analogies of some known classical theorems in complex analysis theory have been developed. It has found many interesting applications in operator theory \cite{ColomboSabadiniStruppa2011}. This new theory contains polynomials and power series with quaternionic coefficients, contrary to the Fueter theory of regular functions defined by means of the Cauchy-Riemann-Fueter differential operator.
  The meeting point between the two function theories comes from an idea of Fueter in the thirties and next developed later by Sce and by Qian.
 This connection holds in any odd dimension (and in quaternionic case) and has been explained in \cite{ColomboSabadiniSommen2010} in the language of slice regular functions with values in the quaternions and slice monogenic functions with values in a Clifford algebra.
 The inverse map has been studied in \cite{ColomboSabadiniSommen2011} and still holds in any odd dimension. 
To make the paper self-contained, we review in this section some basic mathematical concepts relevant to slice regular functions. For more details, we refer the readers to  \cite{GentiliStruppa06,GentiliStruppa07,ColomboSabadiniStruppa2011,GentiliStoppatoStruppa2013}.

 Let $\Hq$ denote the quaternion algebra with its standard basis $\lbrace{1,i,j,k}\rbrace$ satisfying the Hamiltonian multiplication $i^2=j^2=k^2=ijk=-1$, $ij=-ji=k$, $jk=-kj=i$ and $ki=-ik=j$.
 For $q\in{\Hq}$, we write $q=x_0+x_1i+x_2j+x_3k$
 with $x_0,x_1,x_2,x_3\in{\mathbb{R}}$. With respect to the quaternionic conjugate defined to be $\overline{q}=x_0-x_1i-x_2j-x_3k=Re(q)-Im(q)$, we have
 $\overline{ pq }= \overline{q}\, \overline{p}$ for $p,q\in \Hq$. The modulus of $q$ is defined to be $\vert{q}\vert=\sqrt{q\overline{q}}=\sqrt{x_0^2+x_1^2+x_2^2+x_3^2}$. In particular, we have $\vert{Im(q)}\vert=\sqrt{x_1^2+x_2^2+x_3^2}$.
 Notice for instance that the unit sphere $S^2=\lbrace{q\in{Im\Hq}; \vert{Im(q)}\vert=1}\rbrace$ in $Im\Hq$ can be identified with $\mathbb{S}=\lbrace{q\in{\Hq};q^2=-1}\rbrace$, the set of imaginary units.
  Moreover, any $q\in \Hq\setminus \R$ can be rewritten in a unique way as $q=x+I y$ for some real numbers $x$ and $y>0$, and imaginary unit $I\in \mathbb{S}$.

 For every given $I\in{\mathbb{S}}$, we define the slice $L_I$, denoted also $\C_I$, to be $L_I=\C_I = \mathbb{R}+\mathbb{R}I.$
It is isomorphic to the complex plane $\C$ so that it can be considered as a complex plane in $\Hq$ passing through $0$, $1$ and $I$. Their union is the space of quaternions, $\Hq=\underset{I\in{\mathbb{S}}}{\cup}L_I =\underset{I\in{\mathbb{S}}}{\cup}(\mathbb{R}+\mathbb{R}I)$ and their intersection is the real line
$\mathbb{R}=\underset{I\in{\mathbb{S}}}{\cap}L_I =\underset{I\in{\mathbb{S}}}{\cap}(\mathbb{R}+\mathbb{R}I).$

The basic notion in this section is the following

\begin{defn}
A real differentiable function $f: \Omega \longrightarrow \Hq$, with respect to $x_{\ell}$, $\ell = 0, 1, 2, 3$, on a given open domain $\Omega\subset \Hq$, is said to be a slice (left) regular function if, for very $I\in \Sq$, the restriction $f_I$ to $L_{I}=\R+I\R$, with variable $q=x+Iy$, is holomorphic on $\Omega_I := \Omega \cap L_I$, that is it has continuous partial derivatives with respect to $x$ and $y$ and the function
$\overline{\partial_I} f : \Omega_I \longrightarrow \Hq$ defined by
$$
\overline{\partial_I} f(x+Iy):=
\dfrac{1}{2}\left(\frac{\partial }{\partial x}+I\frac{\partial }{\partial y}\right)f_I(x+yI)
$$
vanishes identically on $\Omega_I$.
\end{defn}

Since no confusion can arise, we will refer to slice left-regular functions as slice regular functions or simply regular functions any short and denote their space by $\mathcal{SR}(\Omega)$. It turns out that $\mathcal{S}\mathcal{R}(\Omega)$ is a right vector space over the noncommutative field $\Hq$.
 According to the previous definition, the basic polynomials with quaternionic coefficients on the right are slice regular functions. Notice also that the power series $\sum_n q^na_n$; $a_n\in\Hq$, defines a slice regular function in its domain of convergence, which is proved to be
 an open ball $B(0,R):= \{q\in \Hq; \, |q| < R\}$. Here the space of slice regular functions is endowed with the natural uniform convergence on compact sets.

     Characterization of slice regular functions on a ball $B = B(0,R)$ centered at the origin is given in \cite{GentiliStruppa07}.
Namely, we have

\begin{lemma}[Series expansion]
A given $\Hq$-valued function $f$ is slice regular on $B(0,R)\subset \Hq$ if and only if it has a series expansion of the form:
$$f(q)=\sum_{n=0}^{+\infty} \frac{q^{n}}{n!}\frac{\partial^{n}f}{\partial x^{n}}(0),$$
converging on $B(0,R)=\{q\in\Hq;\mid q\mid< R\}$.
\end{lemma}

\begin{defn}
A domain $\Omega\subset \Hq$ is said to be a slice domain (or just $s$-domain) if  $\Omega\cap{\mathbb{R}}$ is nonempty and for all $I\in{\mathbb{S}}$, the set $\Omega_I:=\Omega\cap{L_I}$ is a domain of the complex plane $L_I:=\mathbb{R}+\mathbb{R}I$.
If moreover, for every $q=x+yI\in{\Omega}$, the whole sphere $x+y\mathbb{S}:=\lbrace{x+yJ; \, J\in{\mathbb{S}}}\rbrace$
is contained in $\Omega$, we say that  $\Omega$ is an axially symmetric slice domain.
\end{defn}

For example the whole space $\Hq$  and the Euclidean ball $B=B(0,R)$ of radius $R$ centered at the origin are axially symmetric slice domains.
The following results are of particular interest (see \cite{ColomboSabadiniStruppa2011, GentiliStoppatoStruppa2013}).

\begin{lemma}[Splitting lemma]\label{split} Let $f:B\longrightarrow{\Hq}$ be a slice regular function.  For every $I$ and $J$ two perpendicular imaginary units, there exist two holomorphic functions $F,G:B_{I}\longrightarrow{L_I}$ such that for all $z=x+yI\in{B_I}$, we have
$$f_I(z)=F(z)+G(z)J,$$
 where $B_I=B\cap{L_I}$ and $L_I=\mathbb{R}+\mathbb{R}I.$
\end{lemma}

\begin{lemma}[Representation formula]\label{repform}
Let $\Omega$ be an axially symmetric slice domain and $f\in{\mathcal{SR}(\Omega)}$. Then, for any $I,J\in{\mathbb{S}}$, we have the formula
$$
f(x+yJ)= \frac{1}{2}(1-JI)f_J(x+yI)+ \frac{1}{2}(1+JI)f_J(x-yI)
$$
for all $q=x+yJ\in{\Omega}$.
\end{lemma}

\begin{lemma}[Extension Lemma]\label{extensionLem}
Let $\Omega_I=\Omega\cap \C_I$; $I\in \mathbb{S}$, be a symmetric domain in $\C_I$ with respect to the real axis and $h:\Omega_I\longrightarrow \Hq$ a holomorphic function. Then, the function $ext(h)$ defined by
$$ext(h)(x+yJ):= \dfrac{1}{2}[h(x+yI)+h(x-yI)]+\frac{JI}{2}[h(x-yI)-h(x+yI)];  \quad J\in \mathbb{S},$$
extends $h$ to a regular function on $\overset{\sim}\Omega=\underset{x+yJ\in{\Omega}}\cup x+y\mathbb{S}$, the symmetric completion of $\Omega_I$.
Moreover, $ext(h)$ is the unique slice regular extension of $h$.
\end{lemma}

\section{The slice hyperholomorphic Bargmann-Fock space $\mathcal{F}^{2,\nu}_{slice}(\Hq)$}
The authors of \cite{AlpayColomboSabadini2014} have defined the slice hyperholomorphic quaternionic Bargmann-Fock space $\mathcal{F}^{2,\nu}_{I}(\Hq)$, for given $I\in{\mathbb{S}}$ and real $\nu>0$, to be
$$\mathcal{F}^{2,\nu}_{I}(\Hq):=\lbrace{f\in{\mathcal{SR}(\Hq); \,  \int_{\C_I}\vert{f_I(q)}\vert^2 e^{-\nu\vert{q}\vert^2}d\lambda_I(q) <\infty}}\rbrace,$$
 where $f_I = f|_{\C_I}$ and $d\lambda_I(q)=dxdy$ for $q=x+yI$.
The right $\Hq$-vector space $\mathcal{F}^{2,\nu}_{I}(\Hq)$ is endowed with the inner product
\begin{equation}\label{spfg}
\scal{f,g}_{\mathcal{F}^{2,\nu}_{I}(\Hq)} = \int_{\C_I}\overline{g_I(q)}f_I(q)e^{-\nu\vert{q}\vert^2} d\lambda_I(q)
\end{equation}
 for $f,g\in{\mathcal{F}^{2,\nu}_{I}(\Hq)}$, so that the associated norm is given by
      $$\Vert{f}\Vert_{\mathcal{F}^{2,\nu}_{I}(\Hq)}^2= \int_{\C_I}\vert{f_I(q)}\vert^2 e^{-\nu\vert{q}\vert^2}d\lambda_I(q).$$
 It is shown in \cite{AlpayColomboSabadini2014} that the monomials $e_n(q):=q^n$; $n=0,1,2, \cdots,$ form an orthogonal basis of $\mathcal{F}^{2,\nu}_{I}(\Hq)$
 with
 \begin{equation}\label{spmonomials}
 \scal{e_m,e_n}_{\mathcal{F}^{2,\nu}_{I}(\Hq)}  = \frac{\pi m!}{\nu^{m+1}}\delta_{m,n}.
\end{equation}
 Moreover, for $f= \sum\limits_{n=0}^{\infty}e_na_n$ and $g= \sum\limits_{n=0}^{\infty}e_nb_n$ two functions in $\mathcal{F}^{2,\nu}_{I}(\Hq)$,  we have
\begin{equation}\label{spfgexp}
\scal{f,g}_{\mathcal{F}^{2,\nu}_{I}(\Hq)}=\left(\frac{\pi}{\nu}\right) \sum_{n=0}^{\infty}\frac{n!}{\nu^{n}}\overline{b_n}a_n,
\end{equation}
  so that a given series $f(q)= \sum\limits_{n=0}^{\infty}q^na_n$ belongs to $\mathcal{F}^{2,\nu}_{I}(\Hq)$ if and only if the quaternionic sequence $(a_n)_n$ satisfies the growth condition
  \begin{align} \label{GC}
  \Vert{f}\Vert_{\mathcal{F}^{2,\nu}_{I}(\Hq)}^2= \left(\frac{\pi}{\nu}\right) \sum_{n=0}^{\infty}\frac{n!}{\nu^{n}}\vert{a_n}\vert^2<\infty.
  \end{align}
This is used by the authors of \cite{AlpayColomboSabadini2014} to prove that the definition of the slice hyperholomorphic Bargmann-Fock space $\mathcal{F}^{2,\nu}_{I}(\Hq)$ is in fact independent of the choice of the imaginary unit $I\in{\mathbb{S}}$.
The following theorem shows a new direct proof of this fact. More precisely, we have the following

 \begin{theorem}
   For any slice regular function $f$ on $\Hq$ and $I,J\in{\mathbb{S}}$, we have
  $$ \frac{1}{2}\Vert{f}\Vert_{\mathcal{F}^{2,\nu}_{I}(\Hq)}\leq{\Vert{f}\Vert_{\mathcal{F}^{2,\nu}_{J}(\Hq)}}\leq 2\Vert{f}\Vert_{\mathcal{F}^{2,\nu}_{I}(\Hq)}.
  $$
  Accordingly, the definition of the slice hyperholomorphic quaternionic Bargmann-Fock space $\mathcal{F}^{2,\nu}_{I}(\Hq)$
  does not depend on the choice of the slice $\C_I$.
  \end{theorem}

  \begin{proof}
  Starting from the representation formula (Lemma \ref{repform}), we get
   $$f_J(x+yJ)= \frac{1}{2}(1-JI)f_I(x+yI)+ \frac{1}{2}(1+JI)f_I(x-yI)$$
  for given $I,J\in{\mathbb{S}}$, and therefore
     $$\left|f_J(x+yJ)\right|\leq  \frac{1}{2}\left|1-JI\right| \left|f_I(x+yI)\right|+ \frac{1}{2}\left|1+JI\right| \left|f_I(x-yI)\right|.$$
Since $\left|1 \pm IJ\right|\leq 1+\left|JI\right|\leq 2$,  we obtain
$$\left|f_J(x+yJ)\right| \leq  \left|f_I(x+yI)\right|+ \left|f_I(x-yI)\right|,$$
and therefore
\begin{align*}
\vert{f_J(x+yJ)}\vert^2 &\leq  \left(\vert{f_I(x+yI)}\vert + \vert{f_I(x-yI)}\vert\right)^2
\\& \leq 2\left(\vert{f_I(x+yI)}\vert^2+\vert{f_I(x-yI)}\vert^2\right)
\end{align*}
because $\left(\vert{f_I(x+yI)}\vert - \vert{f_I(x-yI)}\vert\right)^2 \geq 0$.
This implies that
\begin{align*}
\Vert{f}\Vert^2_{\mathcal{F}^{2,\nu}_{J}(\Hq)}
&=\int_{\C_J} \vert{f_J(x+yJ)}\vert^2e^{-\nu(x^2+y^2)} dxdy
 \\& \leq 2 \int_{\C_I} \vert{f_I(x+yI)}\vert^2e^{-\nu(x^2+y^2)} dxdy
  \\ & \qquad + 2 \int_{\C_I} \vert{f_I(x-yI)}\vert^2e^{-\nu(x^2+y^2)} dxdy
 \\& \leq 2\left(\Vert{f}\Vert^2_{\mathcal{F}^{2,\nu}_{I}(\Hq)}+\Vert{f}\Vert^2_{\mathcal{F}_{\nu,-I}^2(\Hq)}\right).
\end{align*}
  But since $\Vert{f}\Vert_{\mathcal{F}^{2,\nu}_{I}(\Hq)}=\Vert{f}\Vert_{\mathcal{F}_{\nu,-I}^2(\Hq)}$, we conclude that $\Vert{f}\Vert^2_{\mathcal{F}^{2,\nu}_{J}(\Hq)}\leq 4\Vert{f}\Vert^2_{\mathcal{F}^{2,\nu}_{I}(\Hq)}$.
  Now, by interchanging the roles of $I$ and $J$ we get also $\Vert{f}\Vert^2_{\mathcal{F}^{2,\nu}_{I}(\Hq)}\leq 4\Vert{f}\Vert^2_{\mathcal{F}^{2,\nu}_{J}(\Hq)}$. Thus, it follows that  $$\frac{1}{2}\Vert{f}\Vert_{\mathcal{F}^{2,\nu}_{I}(\Hq)}\leq{\Vert{f}\Vert_{\mathcal{F}^{2,\nu}_{J}(\Hq)}}\leq 2\Vert{f}\Vert_{\mathcal{F}^{2,\nu}_{I}(\Hq)}.$$
From the last inequality, we see easily that if $f\in{\mathcal{F}^{2,\nu}_{I}(\Hq)}$ then $f\in{\mathcal{F}^{2,\nu}_{J}(\Hq)}$ and vice versa. Therefore, $\mathcal{F}^{2,\nu}_{I}(\Hq)=\mathcal{F}^{2,\nu}_{J}(\Hq)$ and the space $\mathcal{F}^{2,\nu}_{I}(\Hq)$ is independent of the choice of the slice $\C_I$. This completes the proof.
  \end{proof}

\begin{remark}
According to the previous result, we will denote $\mathcal{F}^{2,\nu}_{I}(\Hq)$ simply by $\mathcal{F}^{2,\nu}_{slice}(\Hq)$.
\end{remark}

The next obvious result seems to be known \cite{AlpayColomboSabadini2014}. However, we did not locate any proof of it in the literature. For the sake of completeness we present here a direct one.

  \begin{proposition}
  The slice hyperholomorphic Bargmann-Fock space $\mathcal{F}^{2,\nu}_{slice}(\Hq)$ is a (right) quaternionic Hilbert space.
  \end{proposition}

  \begin{proof}
 What is needed to prove this theorem is completeness.  For this, let $(f_n)_n$ be a Cauchy sequence in $\mathcal{F}^{2,\nu}_{slice}(\Hq)$ and choose $J\in \Sq$ to be orthogonal to $I$. By the splitting lemma (Lemma \ref{split}) for slice regular functions, we can write every $f_n$ on the slice $\C_I$ as
$${f_n}_{I}=F_n + G_n J.$$
It is clear that $(F_n)_n$ and $(G_n)_n$ are Cauchy sequences of $\C_I$-valued holomorphic functions in the standard Bargmann-Fock Hilbert space $\mathcal{F}^{2,\nu}(\C_I)$ on $\C_I$, to wit
$$\mathcal{F}^{2,\nu}(\C_I):=\mathcal{H}ol(\C_I)\cap L^{2,\nu}(\C_I; e^{-\nu|z|^2} d\lambda_I(z)).$$
 This follows thanks to the well-established fact
 $$\vert{ {f_n}_{I} }\vert^2=\vert{F_n}\vert^2+\vert{G_n}\vert^2.$$
 Therefore, $(F_n)_n$ (resp. $(G_n)_n$) converges in $\mathcal{F}^{2,\nu}(\C_I)$ to some unique $F$ (resp. $G$) belonging to $\mathcal{F}^{2,\nu}(\C_I)$.
Hence, by the extension lemma (Lemma \ref{extensionLem}), the holomorphic function $F$ (resp. $G$) on $\C_I$ can be
extended to a slice regular function on the whole $\Hq$, say $ext(F)$ (resp. $ext(G)$).
Thus, $f:=ext(F) + ext(G)J$ is slice regular.
 Moreover, we have $f_I = F + GJ$ and therefore
 $$
 \Vert{f_n-f}\Vert^2_{\mathcal{F}^{2,\nu}_{slice}(\Hq)}
 \leq \Vert{{F_n} - F}\Vert^2_{\mathcal{F}^{2,\nu}(\C_I)} + \Vert{{G_n}   - G}\Vert^2_{\mathcal{F}^{2,\nu}(\C_I)}
 $$
 tends to $0$ as $n$ goes to $+\infty$. This shows that $(f_n)_n$ converges in $\mathcal{F}^{2,\nu}_{slice}(\Hq)$ to $f$ that also belongs to $\mathcal{F}^{2,\nu}_{slice}(\Hq)$. This ends the proof.
  \end{proof}

Notice also that for fixed $q\in{\Hq}$, the evaluation map $\delta_q : \mathcal{F}^{2,\nu}_{slice}(\Hq) \longrightarrow  \Hq $; $\delta_q(f):= f(q)$, is a continuous linear form. More precisely, we have

\begin{lemma}
For every $f\in \mathcal{F}^{2,\nu}_{slice}(\Hq)$, we have the estimation
\begin{align*}
\vert{\delta_q(f)}\vert \leq \left(\frac{\nu}{\pi}\right)^{1/2}\exp{\left(\frac{\nu}{2}\vert{q}\vert^2\right)} \Vert{f}\Vert_{\mathcal{F}^{2,\nu}_{slice}(\Hq)}.
\end{align*}
\end{lemma}

\begin{proof}
This follows easily by expanding $f$ as series, $f(q)= \sum\limits_{n=0}^{\infty}q^na_n$, and next making use of the Cauchy-Schwartz's inequality,
taking into account \eqref{GC}. Indeed,
\begin{align*}
\vert{f(q)}\vert & \leq{ \sum_{n=0}^{\infty}\vert{q}\vert^n\vert{a_n}\vert}
                   \leq \sum_{n=0}^{\infty} \left(\frac{\sqrt{\nu^{n+1}}}{\sqrt{\pi n!}}\vert{q}\vert^n \right) \left(\frac{\sqrt{\pi n!}}{\sqrt{\nu^{n+1}}}\vert{a_n}\vert\right)
\\& \leq  \left( \sum_{n=0}^{\infty}\frac{\nu (\nu \vert{q}\vert^2)^n}{\pi n!}\right)^{\frac{1}{2}}\left( \sum_{n=0}^{\infty}\frac{\pi n!}{\nu^{n+1}}\vert{a_n}\vert^2\right)^\frac{1}{2}
\\&  \leq \left(\frac{\nu}{\pi}\right)^{1/2}\exp{\left(\frac{\nu}{2}\vert{q}\vert^2\right)} \Vert{f}\Vert_{\mathcal{F}^{2,\nu}_{slice}(\Hq)}.
\end{align*}
\end{proof}

 Thus, by Riesz' representation theorem for quaternionic Hilbert spaces (\cite[Theorem 1]{TobarMandic2014}), there exists a unique element $K_q^{\nu}$ in $\mathcal{F}^{2,\nu}_{slice}(\Hq)$ such that:
 $$\scal{f,K_q^{\nu}}_{\mathcal{F}^{2,\nu}_{slice}(\Hq)}=\delta_q(f)=f(q)$$ for all $f\in \mathcal{F}^{2,\nu}_{slice}(\Hq).$
The reproducing kernel function $K_{\nu} : \Hq\times{\Hq}  \longrightarrow   \Hq$; $(p,q)  \longmapsto  K_{\nu}(p,q)=K_q^{\nu}(p)$ is then given by
\begin{equation}\label{expRK}
K_{\nu}(p,q)  = \left(\frac{\nu}{\pi}\right) \sum_{n=0}^\infty\frac{\nu^np^n\overline{q}^n}{n!}  = \overline{K_{\nu}(q,p)} .
\end{equation}
Accordingly, we have the reproducing property
$$f(q)= \int_{\C_I}\overline{K_{\nu}(p,q)}f_I(p)e^{-\nu\vert{p}\vert^2}d\lambda_I(p) = \scal{ f, K_{\nu}(\cdot,q)}_{\mathcal{F}^{2,\nu}_{slice}(\Hq)}.$$
for every given $I\in{\mathbb{S}}$ and $f\in {\mathcal{F}^{2,\nu}_{slice}(\Hq)}$. Moreover, we prove the following

\begin{proposition} For every $q,q'\in{\Hq}$, we have
$$\scal{K_q^{\nu},K_{q'}^{\nu}}_{\mathcal{F}^{2,\nu}_{slice}(\Hq)} = {K_{\nu}(q',q)}$$
and in particular
$$ \norm{K_q^{\nu}}_{\mathcal{F}^{2,\nu}_{slice}(\Hq)}^2= \left( \frac{\nu}{\pi}\right)e^{\nu |q|^2}.$$
\end{proposition}

\begin{proof} The above assertion follows easily making use of the reproducing property
$f(q') = \scal{ f, K_{\nu}(\cdot,q')}_{\mathcal{F}^{2,\nu}_{slice}(\Hq)}$ applied to
$f(p):= K_{\nu}(p,q)= K_q^{\nu}(p)$ belonging to $\mathcal{F}^{2,\nu}_{slice}(\Hq)$.
As immediate consequence, we obtain
$$
\norm{K_q^{\nu}}_{\mathcal{F}^{2,\nu}_{slice}(\Hq)}^2 =  \scal{K_q^{\nu},K_{q}^{\nu}} = K_{\nu}(q,q) =  \left( \frac{\nu}{\pi}\right)e^{\nu |q|^2}.
$$
\end{proof}

\begin{remark} The previous fact can also be handled using  \eqref{spfgexp} with the expansions
$$ f(p)= K_q^{\nu}(p)= \sum_{n=0}^{\infty}p^na_n \quad \mbox{and} \quad  g(q) = K_{q'}^{\nu}(p)= \sum_{n=0}^{\infty}p^nb_n$$
for every $p\in{\Hq}$ and fixed $q,q'\in \Hq$. Here
$$a_n=  \left(\frac{\nu}{\pi}\right) \frac{\nu^n\overline{q}^n}{n!} \quad \mbox{and} \quad b_n= \left(\frac{\nu}{\pi}\right) \frac{\nu^n\overline{q'}^n}{n!}$$
according to the expansion \eqref{expRK}.
\end{remark}
\section{A quaternionic analogue of the Segal-Bargmann transform}

In this section we propose a quaternionic analogue of the Segal-Bargmann transform. We show that it realizes, as in the complex setting, an isometric isomorphism from the Hilbert space $L^2(\mathbb{R};dx)=L^2(\mathbb{R};\Hq)$, consisting of all the square integrable $\Hq$-valued functions with respect to
\begin{align}\label{spR1}
\scal{\varphi,\psi}_{L^2(\mathbb{R};dx)} : =  \int_{\R} \overline{\psi(x)} \varphi(x) dx,
\end{align}
 onto the slice hyperholomorphic Bargmann-Fock space $\mathcal{F}^{2,\nu}_{slice}(\Hq)$. To this end, we consider the kernel function
\begin{equation}\label{KerneFct}
A(q;x) :=   \left(\frac{\nu}{\pi}\right)^{3/4} e^{\frac{-\nu}{2}(q^2+x^2)+\nu \sqrt{2}qx}; \quad (q,x)\in{\Hq\times{\mathbb{R}}},
\end{equation}
obtained as the slice hyperholomorphic extension of the kernel function of the classical Segal-Bargmann transform.
This is closely connected with the fact that $A(q;x)$ can be seen as the generating function of the real weighted Hermite functions
$$ h_n^\nu(x) := (-1)^n e^{\frac{\nu}{2}x^2} \frac{d^n}{dx^n}\left(e^{-\nu x^2}\right) $$
that form an orthogonal basis of $L^2(\mathbb{R};dx)$, with norm given explicitly by
\begin{equation}\label{normhn}
\norm{h_n^\nu}_{L^2(\mathbb{R};dx)}^2=  2^n\nu^n n!\left(\frac{\pi}{\nu}\right)^{1/2} .
\end{equation}
Namely, we assert

  \begin{proposition} \label{GenFct}
Set $e_n(q):=q^n$. Then for every $q\in{\Hq}$ and $x\in \R$, we have
$$  A(q;x) = \sum_{n=0}^\infty\frac{h_n^\nu(x)}{\norm{h_n^\nu}_{L^2(\mathbb{R};dx)}}  \frac{e_n(q)}{\norm{e_n}_{\mathcal{F}^{2,\nu}_{slice}(\Hq)}}.$$
 \end{proposition}

  \begin{proof}
  By means of the explicit expressions of the norms of $h_n^\nu$ (see \eqref{normhn}) and $e_n$ (see \eqref{spmonomials}) combined with the standard generating function for the Hermite polynomials (see for example \cite{Lebedev1972}), we can check that
 \begin{align*} \sum_{n=0}^\infty\frac{h_n^\nu(x)}{\norm{h_n^\nu}_{L^2(\mathbb{R};dx)}}  \frac{q^n}{\norm{e_n}_{\mathcal{F}^{2,\nu}_{slice}(\Hq)}}
 = \left(\frac{\nu}{\pi}\right)^{3/4} \sum_{n=0}^\infty\frac{h_n^\nu(x) q^n}{n!2^{\frac{n}{2}}}
 A(q;x)
 \end{align*}
 for every given $(q;x)\in{\Hq\times{\mathbb{R}}}$.
   \end{proof}

Another property concerns the partial function of the above kernel function and defined on $\R$ by $A_q:x\mapsto{A_q(x):=A(q;x)}$ for every fixed $q\in \Hq$. It connects the norm of $A_q$ in $L^2(\mathbb{R};dx)$ to the one of the reproducing kernel function $K^{\nu}_{q}$ in $\mathcal{F}^{2,\nu}_{slice}(\Hq)$.
In fact, we have

  \begin{proposition} \label{AqKq}
For every fixed $q\in{\Hq}$, the function $A_q$ is an element of $L^2(\mathbb{R};dx)$ and satisfies
\begin{equation}\label{normAqKq}
\norm{A_q}_{L^2(\mathbb{R};dx)}= \left(\frac{\nu}{\pi}\right)^{1/2} e^{\frac{\nu}{2}\left|{q}\right|^{2}}=\norm{K^{\nu}_{q}}_{\mathcal{F}^{2,\nu}_{slice}(\Hq)}.
 \end{equation}
 \end{proposition}

 \begin{proof}
 By writing the quaternion $q$ as $q=u+vI_q\in{\C_{I_q}}$, with $u,v\in{\mathbb{R}}$ and $I_q\in \mathbb{S}$, we get $q^2=u^2-v^2+2uvI_q$.
 Thus, 
 we can write the modulus of the kernel function $A_q(x)$ as
$$
\left|{A_q(x)}\right|^{2}= \left(\frac{\nu}{\pi}\right)^{\frac{3}{2}} e^{-\nu(u^2-v^2+x^2)+\nu \sqrt{2}ux}.
$$
Therefore, it follows that
 $$
 \int_{\mathbb{R}}\left|{A_q(x)}\right|^{2}dx=\left(\frac{\nu}{\pi}\right)^{\frac{3}{2}} e^{-\nu(u^2-v^2)}\int_{\mathbb{R}}e^{-\nu x^2+\nu \sqrt{2}ux}dx.
 $$
 In the last equality, we recognize the Gaussian integral given by
 \begin{equation}\label{Gaussianintegral}
 \int_{\mathbb{R}}e^{-ax^2+bx}dx=\left(\frac{\pi}{a}\right)^{\frac{1}{2}}e^{\frac{b^2}{4a}}
 \end{equation}
 with $a=\nu$ and $b=\nu\sqrt{2}u$. This implies
 $$
 \norm{A_q}_{L^2(\mathbb{R};dx)}=\left(\frac{\nu}{\pi}\right)^{1/2} e^{\frac{\nu}{2}\left|{q}\right|^{2}}=\norm{K^{\nu}_{q}}_{\mathcal{F}^{2,\nu}_{slice}(\Hq)}.
 $$
 \end{proof}

 Associated to the kernel function $A(q;x)$ given through \eqref{KerneFct}, we consider the integral transform defined by
 \begin{align} \label{defQSBT}
 \mathcal{B}_\nu (\psi)(q)= \int_{\mathbb{R}}A(q;x)\psi(x)dx
 = \left(\frac{\nu}{\pi}\right)^{\frac{3}{4}} \int_{\mathbb{R}}e^{\frac{-\nu}{2}(q^2+x^2)+\nu \sqrt{2}qx}\psi(x)dx
\end{align}
 for $q\in{\Hq}$ and $\psi: \mathbb{R}\longrightarrow{\Hq}$, provided that the integral exists.
 We will call it the quaternionic Segal-Bargmann transform. According to Proposition \ref{GenFct}, we get
 the analogue of Bargmann's result for the classical Segal-Bargmann transform \cite{Bargmann1961}. Namely, we have
   \begin{align*}
   \mathcal{B}_\nu (\psi)(q) &= \sum_{n=0}^\infty e_n(q) \frac{\scal{\psi,h_n^\nu}_{L^2(\mathbb{R};dx)}} {\norm{h_n^\nu}_{L^2(\mathbb{R};dx)} \norm{e_n}_{\mathcal{F}^{2,\nu}_{slice}(\Hq)}}
   \\&= \left(\frac{\nu}{\pi}\right)^{3/4} \sum_{n=0}^\infty\frac{q^n } { 2^{n/2} n! } \scal{\psi,h_n^\nu}_{L^2(\mathbb{R};dx)} .
    \end{align*}
 The following shows that $\mathcal{B}_\nu$ is well defined on $L^2(\mathbb{R};dx)$.

  \begin{proposition} \label{19}
For every $q\in{\Hq}$ and every $\psi \in L^2(\mathbb{R};dx)$, we have
\begin{equation}\label{WellDefined}
 \left|{\mathcal{B}_\nu (\psi)(q)}\right|\leq \left(\frac{\nu}{\pi}\right)^{1/2}e^{\frac{\nu}{2}\vert{q}\vert^2}\Vert{\psi}\Vert_{L^2(\mathbb{R};dx)}.
 \end{equation}
 \end{proposition}

 \begin{proof}
In view of \eqref{normAqKq}, the inequality \eqref{WellDefined} reads simply
$$\left|{\mathcal{B}_\nu (\psi)(q)}\right|\leq{\norm{A_q}_{L^2(\mathbb{R};dx)}}\norm{\psi}_{L^2(\mathbb{R};dx)}$$
which follows immediately by means of
$$
\left|{\mathcal{B}_\nu (\psi)(q)}\right| \leq  \int_{\mathbb{R}} \left|{A(q;x)}\right| \left|{\psi(x)}\right| dx
$$
combined with the Cauchy Schwarz inequality.
\end{proof}

The following gives the explicit expression of the Segal-Bargmann transform acting on the Hermite functions $h_n^\nu$.
Namely, we have

  \begin{lemma} \label{action}
 For every quaternion $q\in{\Hq}$ and nonnegative integer $n$, we have
 $$ \mathcal{B}_\nu (h_n^\nu)(q)=   \left(\frac{\nu}{\pi}\right)^{\frac{1}{4}}2^{\frac{n}{2}}\nu^nq^n$$
 and
 $$ \norm{\mathcal{B}_\nu (h_{n}^\nu)}_{\mathcal{F}^{2,\nu}_{slice}(\Hq)} =\norm{h_n^\nu}_{L^2(\mathbb{R};dx)}.$$
 \end{lemma}

 \begin{proof}
 By definition of $\mathcal{B}_\nu$, we can write
\begin{align*}
   \mathcal{B}_\nu (h_n^\nu)(q) 
   & =  \left(\frac{\nu}{\pi}\right)^{\frac{3}{4}}\int_{\mathbb{R}}e^{\frac{-\nu}{2}(q^2+x^2)+\nu \sqrt{2}qx} (-1)^{n}e^{\frac{\nu}{2} x^2}\frac{d^{n}}{dx^{n}}\left(e^{-\nu x^2}\right)dx \\
     & =  \left(\frac{\nu}{\pi}\right)^{\frac{3}{4}}(-1)^{n}e^{\frac{-\nu}{2} q^2}
     \int_{\mathbb{R}}e^{\nu \sqrt{2}qx}  \frac{d^{n}}{dx^{n}}\left(e^{-\nu x^2}\right)dx.
\end{align*}
Integration by parts yields
\begin{align*}
\mathcal{B}_\nu (h_n^\nu)(q) &= \left(\frac{\nu}{\pi}\right)^{\frac{3}{4}}(-1)^{n-1}e^{\frac{-\nu}{2} q^2}
\int_{\mathbb{R}}\frac{d}{dx} \left(e^{\nu \sqrt{2}qx}\right) \frac{d^{n-1}}{dx^{n-1}}\left(e^{-\nu x^2}\right)dx\\
&= \left(\frac{\nu}{\pi}\right)^{\frac{3}{4}}\nu \sqrt{2}q \int_{\mathbb{R}}e^{\frac{-\nu}{2}(q^2+x^2)+\nu \sqrt{2}qx} (-1)^{n-1}e^{\frac{\nu}{2} x^2}\frac{d^{n-1}}{dx^{n-1}}\left(e^{-\nu x^2}\right)dx .
\end{align*}
This shows that  $\mathcal{B}_\nu (h_n^\nu)(q)=\nu \sqrt{2}q \mathcal{B}_\nu (h_{n-1}^\nu)(q)$ and therefore
$$
\mathcal{B}_\nu (h_n^\nu)(q)= \left(\nu \sqrt{2}q\right)^n \mathcal{B}_\nu (h_{0}^\nu)(q)
 $$
holds by induction. Now, making use of \eqref{Gaussianintegral}, we obtain
$$
\mathcal{B}_\nu (h_n^\nu) (q) =  \left(\nu \sqrt{2}q\right)^n \left(\frac{\nu}{\pi}\right)^{\frac{3}{4}}e^{-\frac{\nu}{2}q^2}\int_{\mathbb{R}}e^{-\nu x^2+\nu \sqrt{2}qx}dx
  = \left(\frac{\nu}{\pi}\right)^{\frac{1}{4}} 2^{\frac{n}{2}} \nu^n q^n.
$$
Using this, it follows
 \begin{align*}
\norm{\mathcal{B}_\nu (h_n^\nu)}_{\mathcal{F}^{2,\nu}_{slice}(\Hq)}^2
= \left(\frac{\nu}{\pi}\right)^{\frac{1}{2}} 2^{n} \nu^{2n} \norm{e_n}^2
=  \left(\frac{\nu}{\pi}\right)^{\frac{1}{2}} 2^{n} \nu^{2n} \frac{\pi n!}{\nu^{n+1}}
= \norm{h_n^\nu}_{L^2(\mathbb{R};dx)}^2,
\end{align*}
thanks to the explicit expressions of the square norms  $\norm{e_n}^2_{\mathcal{F}^{2,\nu}_{slice}(\Hq)}$ and $\norm{h_n^\nu}_{L^2(\mathbb{R};dx)}^2$ given by \eqref{spmonomials} and \eqref{normhn}, respectively.
 \end{proof}

\begin{remark}
The previous result can be reworded by saying that the quaternionic Segal-Bargmann transform $\mathcal{B}_\nu $ maps the orthonormal basis  $$\psi_n^\nu(x):= \frac{h_n^\nu(x)}{\norm{h_n^\nu}_{L^2(\mathbb{R};dx)}}= \left(\frac{\nu}{\pi}\right)^{\frac{1}{4}} \frac{h_n^\nu(x)}{2^{\frac{n}{2}}\nu^{n/2}\sqrt{n!}} $$ of
$L^2(\mathbb{R};dx)$
to the orthonormal basis $ \nu^{n/2}\left(\frac{\nu}{\pi n!}\right)^{1/2} e_n$ of the quaternionic Bargmann-Fock space $\mathcal{F}^{2,\nu}_{slice}(\Hq)$.
\end{remark}

  As immediate consequence, we prove the main result of this section.

 \begin{theorem}
 The quaternionic Segal-Bargmann transform $\mathcal{B}_\nu$ realizes a surjective isometry from the Hilbert space $L^2(\mathbb{R};dx)$ onto the slice hyperholomorphic Bargmann-Fock space $\mathcal{F}^{2,\nu}_{slice}(\Hq)$.
 \end{theorem}

 \begin{proof} Notice first that the result
 $$ \norm{\mathcal{B}_\nu (h_n^\nu)}_{\mathcal{F}^{2,\nu}_{slice}(\Hq)} =\norm{h_n^\nu}_{L^2(\mathbb{R};dx)},$$
 in Lemma \ref{action}, can be extended to finite sums of $h_{n}^\nu c_n$ by linearity of $\mathcal{B}_\nu$ and the orthogonality of $\mathcal{B}_\nu(h_{n}^\nu)$ in $\mathcal{F}^{2,\nu}_{slice}(\Hq)$.  That is
 \begin{align}
    \norm{\mathcal{B}_\nu\left(\sum\limits_{n=N+1}^{N+k} h_n^\nu c_n \right) }_{\mathcal{F}^{2,\nu}_{slice}(\Hq)}^2
     &=\sum\limits_{n=N+1}^{N+k} |c_n|^2 \norm{ h_n^\nu }_{L^2(\mathbb{R};dx)}^2   \nonumber
    \\& =\norm{\sum\limits_{n=N+1}^{N+k} h_n^\nu c_n }_{L^2(\mathbb{R};dx)}^2 . \label{finiteiso}
    \end{align}
Now, any $\psi\in{L^2(\mathbb{R};dx)}$ can be expanded as $\psi= \sum\limits_{n=0}^{\infty} h_n^\nu c_n$; $(c_n)_n \subset {\Hq}$ for $\lbrace{h_k^\nu; k\geq 0}\rbrace$ being a basis of $L^2(\mathbb{R};dx)$.
   The sequence $\psi_N :=  \sum_{n=0}^{N} h_n^\nu c_n$ converges to $\psi$ in $L^2(\mathbb{R};dx)$-norm and therefore
   $(\mathcal{B}_\nu\psi_N)_N$ is a Cauchy sequence in the Hilbert space $\mathcal{F}^{2,\nu}_{slice}(\Hq)$.
    Indeed, from \eqref{finiteiso} we get
    \begin{align*}
    \norm{\mathcal{B}_\nu\psi_N-\mathcal{B}_\nu\psi_M}_{\mathcal{F}^{2,\nu}_{slice}(\Hq)}
    &=\norm{\mathcal{B}_\nu(\psi_N-\psi_M)}_{\mathcal{F}^{2,\nu}_{slice}(\Hq)}
     \\& =\norm{\psi_N-\psi_M}_{L^2(\mathbb{R};dx)} \underset{N,M\rightarrow{\infty}}\longrightarrow 0.
    \end{align*}
Therefore, there exists $f$ in $\mathcal{F}^{2,\nu}_{slice}(\Hq)$ such that $(\mathcal{B}_\nu\psi_N)_N$ converges to $f$ in the norm of $\mathcal{F}^{2,\nu}_{slice}(\Hq)$. Subsequently, there exists a subsequence $(\mathcal{B}_\nu\phi_{N_k})_{k}$ of $(\mathcal{B}_\nu\phi_{N})_{N}$ converging to $f$, almost everywhere. However, thanks to Proposition \ref{19}, we establish the following estimation
 $$
 \left|{\mathcal{B}_\nu\psi_N(q)- \mathcal{B}_\nu\psi(q)}\right|= \left|\mathcal{B}_\nu(\psi_N - \psi)(q) \right|\leq \left(\frac{\nu}{\pi}\right)^{1/2}e^{\frac{\nu}{2}\left|{q}\right|^{2}}\norm{\psi_N-\psi}_{L^2(\mathbb{R};dx)}.
 $$
 Thus, the sequence $(\mathcal{B}_\nu\psi_N)_N$ converges pointwise to $\mathcal{B}_\nu\psi$ and therefore 
 $\mathcal{B}_\nu\psi=f:=\underset{N\rightarrow{\infty}}\lim{\mathcal{B}_\nu (\psi)_N}.$
Accordingly, for every $\psi\in{L^2(\mathbb{R};dx)}$ we have
 \begin{align*}
 \norm{\mathcal{B}_\nu\psi}_{\mathcal{F}^{2,\nu}_{slice}(\Hq)}
 &=\Vert{\underset{N\rightarrow{\infty}}\lim{\mathcal{B}_\nu (\psi)_N}}\Vert_{\mathcal{F}^{2,\nu}_{slice}(\Hq)}
 \\& =\underset{N\rightarrow{\infty}}\lim\norm{\mathcal{B}_\nu (\psi)_N}_{\mathcal{F}^{2,\nu}_{slice}(\Hq)}
 \\& = \underset{N\rightarrow{\infty}}\lim\norm{\psi_N}_{L^2(\mathbb{R};dx)}
 \\& =\norm{\psi}_{L^2(\mathbb{R};dx)}.
 \end{align*}
   This shows that $\mathcal{B}_\nu$ is an isometry from $L^2(\mathbb{R};dx)$ into $\mathcal{F}^{2,\nu}_{slice}(\Hq)$.

 The injectivity is obvious from the isometry property. The proof that $\mathcal{B}_\nu$ is surjective relies on the fact that (Lemma \ref{action})
 $$e_n(q)=
 \left(\frac{\pi}{\nu}\right)^{\frac{1}{4}}\frac{1}{2^{n/2}\nu^n}  \mathcal{B}_\nu (h_n^\nu)(q)=a_{n,\nu}\mathcal{B}_\nu (h_n^\nu)(q).$$
  Indeed, since any $f\in{\mathcal{F}^{2,\nu}_{slice}(\Hq)}$
can be expanded as $f=  \sum_{n=0}^{\infty}e_n c_n$, we get
\begin{align*}
f(q)= \sum_{n=0}^{\infty} a_{n,\nu}\mathcal{B}_\nu (h_n^\nu)(q) c_n= \mathcal{B}_\nu\left( \sum_{n=0}^{\infty}h_n^\nu a_{n,\nu}c_n\right)(q) = \mathcal{B}_\nu (\phi)(q) ,
\end{align*}
where $\phi$ stands for $\phi:= \sum\limits_{n=0}^{\infty} h_n^\nu a_{n,\nu}c_n \in L^2(\mathbb{R};dx)$.
The proof is completed.
 \end{proof}

According to the previous proof, the transform $\mathcal{B}_\nu$ admits an inverse mapping $\mathcal{F}^{2,\nu}_{slice}(\Hq)$ onto $L^2(\mathbb{R};dx)$ and given by the expansion
\begin{align} \label{inverseExp}
[\mathcal{B}_\nu]^{-1}(f)(x) := \sum_{n=0}^{\infty} \frac{\scal{f,e_n}_{\mathcal{F}^{2,\nu}_{slice}(\Hq)}}{\norm{h_n}_{L^2(\mathbb{R};dx)} \norm{e_n}_{\mathcal{F}^{2,\nu}_{slice}(\Hq)} }  h_n(x) =  \left(\frac{\pi}{\nu}\right)^{\frac{1}{4}}  \sum_{n=0}^{\infty} \frac{c_n}{2^{\frac{n}{2}}\nu^{n}}  h_n(x)
\end{align}
for $f := \sum\limits_{n=0}^{\infty} e_n^\nu c_n \in \mathcal{F}^{2,\nu}_{slice}(\Hq)$.
Moreover, one can show the following result.

 \begin{theorem}\label{thm:intrepInverse}
 The inverse transform $[\mathcal{B}_\nu]^{-1}$, mapping unitarily $\mathcal{F}^{2,\nu}_{slice}(\Hq)$ onto $L^2(\mathbb{R};dx)$, is given by the integral representation
 \begin{align} \label{inverseExpIntRep}
[\mathcal{B}_\nu]^{-1}(f)(x) :=
 \left(\frac{\nu}{\pi}\right)^{3/4} \int_{\C_I}  e^{\frac{-\nu}{2}(\overline{q}^2+x^2)+\nu \sqrt{2}\overline{q}x} f(q) e^{-\nu |q|^2} d\lambda_I(q).
\end{align}
 \end{theorem}

  \begin{proof}
$\mathcal{B}_\nu$ is the integral operator associated to the generating function of the Hermite polynomials and can be rewritten as
$$ \mathcal{B}_\nu (\psi)(q) = \scal{ \psi,\overline{A_q}}_{L^2(\mathbb{R};dx)}.$$
This suggests the consideration of the transform
 $$ \widetilde{\mathcal{B}_\nu} (f)(x) := \scal{f,A(\cdot;x)}_{\mathcal{F}^{2,\nu}_{slice}(\Hq)}$$
 for $f\in \mathcal{F}^{2,\nu}_{slice}(\Hq)$. Then, using orthogonality of the monomials in $\mathcal{F}^{2,\nu}_{slice}(\Hq)$,
  it can easily be seen that
  $$ \widetilde{\mathcal{B}_\nu} (e_n)(x) = \frac{\norm{e_n}_{\mathcal{F}^{2,\nu}_{slice}(\Hq)}}{\norm{h_n^\nu}_{L^2(\mathbb{R};dx)}} h_n^\nu(x) =
  \left(\frac{\pi}{\nu}\right)^{1/4} \frac{1}{2^{n/2} \nu^n} h_n^\nu(x).$$
  More generally, by proceeding as for $\mathcal{B}_\nu$ using the linearity of $\widetilde{\mathcal{B}_\nu}$ and the fact that the $h_n^\nu$ take real values, we check that
   \begin{align*}
   \widetilde{\mathcal{B}_\nu} (f)(x) =  \sum\limits_{n=0}^\infty \widetilde{\mathcal{B}_\nu} (e_n)(x) a_n
  = \left(\frac{\pi}{\nu}\right)^{1/4} \sum\limits_{n=0}^\infty \frac{a_n}{2^{n/2} \nu^n} h_n^\nu(x)
   \end{align*}
   for every $f = \sum\limits_{n=0}^\infty e_n a_n$ belonging to $\mathcal{F}^{2,\nu}_{slice}(\Hq)$. Therefore, the expression of $\widetilde{\mathcal{B}_\nu} (f)(x)$ coincides with the one of $[\mathcal{B}_\nu]^{-1}(f)(x)$ given through \eqref{inverseExp}. This completes the proof.
  \end{proof}

  \begin{remark}
  The proof of Theorem \ref{thm:intrepInverse} can be handled directly using the well-established fact
  $\scal{\psi,h_n^\nu}_{L^2(\mathbb{R};dx)} = \alpha_n \norm{h_n^\nu}_{L^2(\mathbb{R};dx)}^2 $
   for every $\psi=\sum\limits_{n=0}^\infty    h_n^\nu \alpha_n\in L^2(\mathbb{R};dx)$. Indeed,
   we have
      \begin{align*}
   \widetilde{\mathcal{B}_\nu} (\mathcal{B}_\nu (\psi))(x)
    & =  \scal{\psi(\cdot) , \scal{A((\cdot\cdot;x), A(\cdot\cdot;\cdot)}_{\mathcal{F}^{2,\nu}_{slice}(\Hq)}} _{L^2(\mathbb{R};dx)}
  \\& =   \sum\limits_{n=0}^\infty \frac{ \scal{\psi,h_n^\nu}_{L^2(\mathbb{R};dx)} }{ \norm{h_n^\nu}_{L^2(\mathbb{R};dx)}^2 } h_n^\nu(x)
   \\& = \psi(x).
   \end{align*}
  \end{remark}

  We conclude this section by showing the connection to the one-dimensional quaternionic Fourier transform \cite{EllLeBihanSangwine2014} defined on $L^1(\R;dx)=L^1(\R;\Hq)$ by
  $$\mathcal{F}_I(\psi)(x) = \int_{\R} e^{I x y} \psi(y) dy $$
  for given $I\in \mathbb{S}$. Indeed, starting from Definition \ref{defQSBT} of $\mathcal{B}_\nu (\psi)$ restricted to $I\R$ we get
   \begin{align*}
 \mathcal{B}_\nu (\psi)\left(\frac{Ix}{\sqrt{2}\nu}\right)
 = \left(\frac{\nu}{\pi}\right)^{\frac{3}{4}}   e^{\frac{x^2}{4\nu}} \int_{\mathbb{R}} e^{\frac{-\nu}{2} y^2+ I xy}\psi(y)dy.
\end{align*}
 for every $\psi \in L^2(\R;dx)$. Thus, we assert

    \begin{proposition} \label{prop:QSBTQFT}
The quaternionic Segal-Bargmann transform is closely related to the left one-dimensional quaternionic Fourier transform. More precisely, for
 any $I \in \mathbb{S}$, real $x\in \R$ and $\psi \in L^2(\R;dx)$, we have
   \begin{align} \label{QSBTQFT}
 \mathcal{B}_\nu (\psi)\left(\frac{Ix}{\sqrt{2}\nu}\right)
 = \left(\frac{\nu}{\pi}\right)^{\frac{3}{4}}   e^{\frac{x^2}{4\nu}} \mathcal{F}_I\left(e^{\frac{-\nu}{2} y^2}\psi\right)(x)
\end{align}
  \end{proposition}

  Moreover, we have the following important result for the special value of $\nu=1$. In this case, we use the notation $\mathcal{B}$ instead of $\mathcal{B}_1$ and $\mathcal{F}^{2}_{slice}(\Hq)$ instead of $\mathcal{F}^{2,1}_{slice}(\Hq)$.

      \begin{theorem}\label{prop:diagcomm}
      For every $f \in \mathcal{F}^{2}_{slice}(\Hq)$, $I \in \mathbb{S}$ and real $ x\in \R$, we have the following identity
          \begin{align} \label{diagcomm}
           \mathcal{B} \mathcal{F}_I \mathcal{B}^{-1} (f)\left(x\right) = \sqrt{2\pi}  f\left(Ix\right).
 \end{align}
 \end{theorem}

  \begin{proof}
  Direct computation shows that
  \begin{align*}
  \mathcal{B}_\nu (\mathcal{F}_I \psi)\left(q\right) =  \sqrt{2}  \left(\frac{\nu}{\pi}\right)^{\frac{1}{4}} e^{\frac{\nu}{2}q^2}
                                             \int_{\R} e^{\sqrt{2}Iqy} \left(e^{\frac{-y^2}{2\nu} }\psi(y)\right) dy
  \end{align*}
  for any $q\in \C_I$. This yields in particular
  \begin{align*}
  \mathcal{B}_\nu (\mathcal{F}_I \psi)\left(\frac{x}{\sqrt{2}}\right)
  =  \sqrt{2}  \left(\frac{\nu}{\pi}\right)^{\frac{1}{4}} e^{\frac{\nu}{4}x^2}\mathcal{F}_I \left(e^{\frac{-y^2}{2\nu} }\psi\right)(x).
  \end{align*}
   Now, by means of Proposition \ref{prop:QSBTQFT} (with $\nu=1$) we get
  \begin{align*}
  \mathcal{B} (\mathcal{F}_I \psi)\left(\frac{x}{\sqrt{2}}\right)
   =  \sqrt{2}  \left(\frac{1}{\pi}\right)^{\frac{1}{4}} e^{\frac{x^2}{4}}  \mathcal{F}_I \left(e^{\frac{-y^2}{2} }\psi\right)(x)
   = \sqrt{2\pi}   \mathcal{B}(\psi)\left(\frac{Ix}{\sqrt{2}}\right).
  \end{align*}
  Finally, the result of Proposition \ref{prop:diagcomm} follows by taking $\psi = \mathcal{B}^{-1} f$ and replacing $x/\sqrt{2}$ by $x$.
  \end{proof}

  \begin{remark}
  Proposition \ref{prop:QSBTQFT} (resp. \ref{prop:diagcomm}) constitutes the quaternionic version of its  well-known  analogous for the classical Segal-Bargmann and Fourier transforms, see for example \cite[Eq. (1.2)]{Hall2001} (resp. \cite[Theorem 3]{ZhuDictionary}).
  \end{remark}
\section{Conclusion}

In this paper we have introduced and studied basic properties of a quaternionic analogue of the classical Segal-Bargmann transform for the slice hyperholomorphic Bargmann-Fock space. We also gave an integral representation of its inverse and established the connection to a one-dimensional quaternionic Fourier transform.
 In a forthcoming work, we plan to provide a systematic study of the full hyperholomorphic Bargmann-Fock space
 $$\mathcal{F}^{2,\nu}_{full}(\Hq):=\mathcal{SR}(\Hq)\cap L^2(\Hq;e^{-\nu\vert{q}\vert^2}d\lambda)$$
 considered as a subspace of the Hilbert space $L^2(\Hq;e^{-\nu\vert{q}\vert^2}d\lambda )$ of all square integrable functions with respect to the Gaussian measure on $\Hq$, $d\lambda$ being the Lebesgue measure on $\Hq=\R^4$.
 The main aim is to suggest a suitable orthogonal basis and an appropriate Segal-Bargmann transform.

\quad

\noindent{\bf Acknowledgement:}
The authors are most grateful to the anonymous referees whose deep and extensive comments greatly contributed to improve this paper.
They also would like to thank Professor Sabadini I. for her important advice and remarks, and Professor Sangwine S.J. for pointing out the book \cite{EllLeBihanSangwine2014}.
The assistance of the members of the seminars ``Partial differential equations and spectral geometry" is gratefully acknowledged.
The second author is partially supported by the Hassan II Academy of Sciences and Technology.

\end{document}